\documentclass[12pt]{amsart}
\usepackage{amssymb}

\newtheorem{thm}     {Theorem}[section]
\newtheorem{prop}    [thm]{Proposition}
\newtheorem{definition}  [thm]{Definition}
\newtheorem{cor}     [thm]{Corollary}
\newtheorem{lemma}   [thm]{Lemma}
\newtheorem{remark}   [thm]{Remark}

\newcommand{\E}{{\mathcal E}}
\newcommand{\F}{{\mathcal F}}
\newcommand{\G}{{\mathcal G}}
\newcommand{\C}{\mathbb C}
\newcommand{\D}{\mathbb D}

\newcommand{\R}{\mathbb R}

\newcommand{\z}{\zeta}
\newcommand{\cod}{{\rm cod\,}}
\newcommand{\id}{{\rm id}}
\newcommand{\Span}{{\rm Span\,}}
\def\d{\partial}
\def\l{\ell}
\def\Re{{\rm Re\,}}

\def\bar{\overline}
\def\<{\langle}
\def\>{\rangle}

\begin{document}


\title[Stationary discs and finite jet determination
for CR mappings]
{Stationary discs and finite jet determination \\
for CR mappings in higher codimension}

\author{Alexander Tumanov}
\address{Department of Mathematics, University of Illinois,
1409 West Green St., Urbana, IL 61801}
\email{tumanov@illinois.edu}

\thanks{Partially supported by Simons Foundation grant.}

\maketitle

\begin{abstract}
We discuss stationary discs for generic CR manifolds and
apply them to the problem of finite jet
determination for CR mappings. We prove that a $C^2$-smooth
CR diffeomorphism of two $C^4$-smooth strictly
pseudoconvex Levi generating CR manifolds is uniquely
determined by its 2-jet at a given point.
A new key element of the proof is the existence of non-defective stationary discs.

\end{abstract}

\section{Introduction}
We discuss stationary discs for generic CR manifolds and  apply them to the problem of finite jet determination for
CR mappings.

Let $M\subset\C^n$ be a generic real submanifold of real
codimension $k$. Then $M$ has CR dimension $m=n-k$.
We introduce coordinates $(z,w)\in\C^n$,
$z=x+iy\in\C^k$, $w\in\C^m$, so that $M$ has a local
equation of the form
$$
x=h(y,w),
$$
where $h=(h_1,\dots,h_k)$ is a smooth real vector function with $h(0)=0$, $dh(0)=0$.
Furthermore, one can choose the coordinates in such
a way that each term in the Taylor expansion of
$h$ contains both $w$ and $\bar w$ variables.
Then the equations of $M$ take the form
$$
x_j=h_j(y,w)=\<A_jw,\bar w\>+O(|y|^3+|w|^3),
\qquad 1\le j \le k.
$$
Here $A_j$ are Hermitian matrices,
and $\<a, b\>=\sum a_\ell b_\ell$.
The matrices $A_j$ can be regarded as the components
of the vector valued Levi form of $M$ at $0$.

\begin{itemize}
\item
We say $M$ is {\em Levi generating} at 0 if the
matrices $A_j$ are linearly independent.
\smallskip

\item
We say $M$ is {\em Levi nondegenerate} at 0 if
$\<A_ju,\bar v\>=0$ for all $j$ and $u$ implies $v=0$.
\smallskip

\item
We say $M$ is {\em strongly Levi nondegenerate} at 0 if
there is $c\in\R^k$ such that
$\det\left(\sum c_j A_j\right)\ne0$.
This condition means that locally $M$ lies on a Levi nondegenerate hypersurface.
\smallskip

\item
We say $M$ is {\em strongly pseudoconvex} at 0 if
there is $c\in\R^k$ such that $\sum c_j A_j >0$.
This condition means that locally $M$ lies on a strognly
pseudoconvex hypersurface.
\end{itemize}

We are concerned with the problem whether a
CR diffeomorphism between two manifolds is uniquely
determined by its finite jet at a point.
This problem is called the {\em finite jet determination},
and it has been a subject of work by many authors
(Baouendi, Beloshapka, Bertrand, Blanc-Centi, Ebenfelt,
Ezhov, Han, Juhlin, Kim, Lamel, Meylan, Mir, Rothschild,
Schmalz, Sukhov, Zaitsev, and others, see \cite{BBM2019}).
Nevertheless, there are fundamental open questions even
in the Levi generating case, to which we restrict here.

For a real analytic Levi nondegenerate hypersurface,
it follows by Chern-Moser \cite{Chern-Moser} that
a real analytic CR automorphism is determined by
its 2-jet at a point.
Beloshapka \cite{Beloshapka} proved that a
real analytic CR automorphism of a real analytic generic
Levi generating and Levi nondegenerate manifold is determined
by its finite jet at a point.
Bertrand, Blanc-Centi, and Meylan \cite{BBM2019, BM2019}
prove
2-jet determination for $C^3$-smooth CR automorphisms of
$C^4$-smooth generic strongly Levi nondegenerate
manifold $M$ under the additional assumption that
there is $v\in \C^m$ such that the vectors
$\{A_j v: 1\le j\le k\}$ are $\R$-linearly independent.
This condition is quite restrictive, in particular,
it implies that $k\le 2m$ whereas the natural restriction
imposed by the Levi generating condition is $k\le m^2$.
On the other hand, Meylan \cite{Meylan} has constructed
an example of a strongly Levi nondegenerate Levi generating
quadric for which 2-jet determination does not hold.
Our main result is the following.

\begin{thm}\label{Main}
Let $M_1$ and $M_2$ be $C^4$-smooth generic
{\em strongly pseudoconvex}
Levi generating manifolds in $\C^n$.
Let $p\in M_1$.
Then every germ at $p$ of a $C^2$-smooth CR diffeomorphism
$f:M_1\to M_2$
is uniquely determined by its 2-jet at $p$.
\end{thm}

Since the result involves 2-jets, the hypothesis that
$f$ is $C^2$-smooth is natural. We recall that under the
hypotheses on $M_1$ and $M_2$, if $f$ is merely $C^1$-smooth,
then it follows that $f$ is automatically Lipschitz
$C^{2,\alpha}$-smooth for all $0<\alpha<1$, see, e. g.,
\cite{T1994}.

Our proof as well as the proofs in \cite{BBM2019, BM2019}
is based on {\em stationary} discs.
Lempert \cite{Lempert} introduced extremal and stationary discs for {\em strongly convex domains} and applied them
to various problems.
The author \cite{T2001} developed a local version of the theory of extremal and stationary discs in {\em higher codimension} and applied it to the regularity of CR mappings.

Stationary discs form a family invariant under CR mappings.
They depend on finitely many, namely $4n$, real parameters.
To apply stationary discs to the problem of finite jet
determination, one would like to know whether such discs
are uniquely determined by their finite jets and whether
they cover a sufficiently large set. This will be the case
if there exists a {\em non-defective} stationary disc
through the given point.

Defective discs arise in the problem of wedge extendibility
of CR functions \cite{T1988} and can be characterized as critical points of the evaluation map $\phi\mapsto\phi(0)$
defined on the set of all complex discs $\phi$ attached
to $M$.
It follows from \cite{T1988} that for a Levi generating manifold $M$, there are many non-defective discs through
every point of $M$. However, it does not follow that there
are non-defective {\em stationary} discs. The key new result
of this paper is the existence of non-defective stationary
discs, which the author conjectured in \cite{T2004} and
reduced it to a linear algebra question. We answer it here
for a strongly pseudoconvex Levi generating manifold.

Using the results obtained here, we can improve our
earlier result of \cite{T2001, T2004} on the regularity of
CR mappings.

\begin{thm}
Let $M_1$ and $M_2$ be $C^\infty$ smooth generic
strictly pseudoconvex Levi generating manifolds in $\C^n$,
and let $F:M_1\to M_2$ be a homeomorphism such that
both $F$ and $F^{-1}$ are CR and satisfy a Lipschitz
condition with some exponent $0<\alpha<1$.
Then $F$ is $C^\infty$ smooth.
\end{thm}

Previously \cite{T2001, T2004}, there was an additional condition on the existence of non-defective stationary
discs for $M_1$ and $M_2$.

The paper is structured as follows.
In Section 2, we recall some basics on CR manifolds.
In Section 3, we introduce stationary discs.
In Section 4, we discuss parametrization of stationary discs
by their jets at a boundary point and prove
Theorem \ref{Main} assuming the needed existence of non-defective stationary discs.
In Section 5, we discuss the existence of non-defective
stationary discs and reduce it to a result in
linear algebra.
In Section 6, we prove that result.

The author wishes to thank the referee for useful
remarks.

\section{CR manifolds}
Let $M$ be a smooth real submanifold in $\C^n$.
Recall that the {\em complex tangent space} at $p\in M$
is the maximum complex subspace in $T_p(M)$.
$$
T^c_p(M) = T_p(M) \cap JT_p(M),\; p\in M.
$$
Here $J : \C^n \to \C^n$ is the operator of
multiplication by $i = \sqrt{-1}$.
The manifold $M$ is called a {\em CR manifold} if
$\dim T^c_p(M)$
does not depend on $p \in M$.
Then the dimension $\dim_\C T^c_p(M)$ is called
the {\em CR dimension} of $M$ and is denoted by
$\dim_{CR} M$.

The manifold $M$ is called {\em generic} if $T_p(M)$
spans  $T_p(\C^n) \simeq \C^n$ over $\C$ for all $p \in M$,
that is,
$$
T_p(M) + JT_p(M) = \C^n.
$$
If $M$ is generic, then $M$ is a CR manifold and
$$
\dim_{CR} M + \cod M = n,
$$
where $\cod M$ is the real codimension of $M$ in $\C^n$.


Let $M$ be a generic manifold in $\C^n$.
Let $T^{*1,0}(\C^n)$ be the bundle of (1,0) forms
in $\C^n$.
$$
T^{*1,0}(\C^n)=\{\omega=\sum\omega_j dz_j:\omega_j\in\C,
1\le j\le n \}.
$$
The {\em conormal} bundle $N^*(M)$ of $M$ in $\C^n$ is the real dual bundle to the normal bundle
$N(M)=T(\C^n)|_M/T(M)$.

Every real form is a real part of a unique (1,0) form.
Then we can view $N^*(M)$ as a real submanifold in $T^{*1,0}(\C^n)$.
$$
N^*_p(M)=\{\omega\in T^{*1,0}_p(\C^n):
\Re\<\omega,X\>=0, X\in T_p(M)\}.
$$

Let $M \subset \C^n$ be a generic manifold of
codimension $k$. Then $M$ has a local equation
$
\rho = 0
$,
where $\rho= (\rho_1,\dots ,\rho_k)$ is a smooth $\R^{k}$-valued
function in a neighborhood of $0\in\C^n$ such that
$\d\rho_1 \wedge \dots \wedge \d  \rho_k \neq 0$.

The spaces $T_p(M)$, $T^c_p(M)$, and $N^*_p(M)$
have the descriptions
\begin{align*}
&T_p(M)=\{X \in T_p(\C^n): \<d\rho_j(p),X\>=0,
1\le j\le k \},\\
&T^c_p(M)=\{X \in T_p(\C^n):
\<\d\rho_j(p),X\> = 0, 1\le j\le k\},\\
&
N^*_p(M)=\{c\d\rho=\sum c_j\d \rho_j(p):c\in\R^k\}.
\end{align*}
Here
$$
d\rho=\d\rho+\bar\d\rho,\quad
d\rho_j=\d\rho_j+\bar\d\rho_j,\quad
\d\rho=\sum \frac{\d\rho}{\d z_\l} dz_\l.
$$

\section{Stationary discs}

A complex disc is a map $\phi:\bar\D\to X$,
of the standard unit disc $\D\subset\C$ to
a complex manifold $X$,
$\phi\in\mathcal{O}(\D)\cap C(\bar\D)$.

We say a complex disc $\phi$ is {\em attached}
to a set $M\subset\C^n$ if $\phi(b\D)\subset M$.

\begin{definition} \cite{Lempert, T2001}
Let $M$ be a generic manifold in $\C^n$.
A complex disc $\phi$ attached to $M$ is
called {\em stationary} if there exists a
nonzero continuous holomorphic mapping
$\phi^*:\bar\D\setminus\{0\}\to T^{*1,0}(\C^n)$,
such that $\zeta\mapsto\zeta \phi^*(\zeta)$
is holomorphic in $\D$ and
$\phi^*(\zeta)\in N^*_{\phi(\zeta)}M$
for all $\zeta\in b\D$.
\end{definition}

In other words, $\phi^*$ is a punctured complex disc
in $T^{*1,0}(\C^n)$
with a pole of order at most one at zero attached to
$N^*(M)\subset T^{*1,0}(\C^n)$ such that the
natural projection sends $\phi^*$ to $\phi$.
We call $\phi^*$ a {\em lift} of $\phi$.
We will always use this meaning of the term ``lift''.

Stationary discs bear this name because they
arise from an extremal problem, but we do not need
it here.

It can happen that a lift has no pole.
Following \cite{T1988},
we call a disc $\phi$ {\em defective} if it has
a nonzero lift $\phi^*$ holomorphic in the whole
unit disc including 0. Defective discs will be the
main concern in this paper.
For a strictly convex hypersurface, all
defective discs are constant. However, this is no
longer true for strictly pseudoconvex manifolds
of higher codimension, see Remark \ref{defective-exist}.

To state the main result on the existence of
stationary discs, we need a general description
of their lifts.
Let $M$ be again a codimension $k$ generic manifold
in $\C^n$ defined by a local equation
\begin{equation}\label{M-rho-h}
\rho=0, \quad
\rho(z,w)=x-h(y,w),
\end{equation}
where
$z=x+iy\in\C^k$,
$w\in\C^m$,
$k+m=n$,
$h(0,0)=0$,
$dh(0,0)=0$.

Let $\phi$ be a complex disc attached to $M$.
We recall a $k\times k$ matrix function
$G$ on $b\D$ such that $G(1)=I$ and
$G\rho_z\circ \phi$
extends holomorphically from $b\D$ to $\D$,
here $\rho_z$ denotes the matrix of partial derivatives
and $I$ is the identity matrix.
For a small disc $\phi$,
such a matrix $G$ close to the identity always exists
and is unique (see \cite{T1988,T2001}).
If $h$ does not depend on $y$, then $G\equiv I$.
For simplicity we will omit writing ``$\circ \phi$''.

\begin{prop} \label{structure-lifts}
\cite{T2001}
Let $\phi$ be a small complex disc attached to $M$.
\begin{itemize}
\item[(i)]
Every lift $\phi^*$ of $\phi$ holomorphic at 0
has the form
$\phi^*|_{b\D}=cG\partial\rho$, where $c\in\R^k$.

\item[(ii)]
$\phi$ is {\em defective} if and only if there exists a
nonzero $c\in\R^k$ such that
$cGh_w$ extends holomorphically to $\D$.

\item[(iii)]
Every lift $\phi^*$ of $\phi$ has the form
$\phi^*|_{b\D}=\Re(\lambda\zeta+c)G\partial\rho$,
where $\lambda\in\C^k$, $c\in\R^k$.

\item[(iv)]
$\phi$ is {\em stationary} if and only if there exist
$\lambda\in\C^k$ and $c\in\R^k$ such that
$\zeta\Re(\lambda\zeta+c)Gh_w$ extends holomorphically
to $\D$.
\end{itemize}
\end{prop}

We again assume that
\begin{equation}\label{hjAj}
h_j(y,w)=\<A_jw,\bar w\>+O(|y|^3+|w|^3),
\quad 1\le j \le k.
\end{equation}
Here is the main result on the existence of
stationary discs (\cite{T2001}, Theorem 5.3).

\begin{thm}\label{existence-stationary}
Let $M\subset\C^n$ be a $C^4$-smooth
strictly pseudoconvex Levi generating manifold
defined by (\ref{M-rho-h}, \ref{hjAj}).
Then for every $\epsilon>0$ there exists $\delta>0$
such that for every
$\lambda\in\C^k$,
$c\in\R^k$,
$w_0\in\C^m$,
$y_0\in\R^k$,
$v\in\C^m$
such that
$$
\sum\Re(\lambda_j\z+c_j)A_j > \epsilon(|\lambda|+|c|)I
$$
for all $\z\in b\D$
and
$|w_0|<\delta,
|y_0|<\delta,
|v|<\delta$
there exists a unique stationary disc
$\z\mapsto \phi(\z)=(z(\z),w(\z))$
with lift $\phi^*$
such that
$w(1)=w_0$,
$w'(1)=v$,
$y(1)=y_0$,
and
$\phi^*|_{b\D}=\Re(\lambda\zeta+c)G\partial\rho$.
The pair $(\phi,\phi^*)$ depends $C^2$-smoothly
on $\zeta\in\bar\D$ and all the parameters
$\lambda, c, w_0, y_0, v$.
\end{thm}

\section{Parametrization by $T(N^*(M))$ and proof of the main result}

Let $\E$ be the set of all stationary discs with
lifts $(\phi,\phi^*)$ provided by
Theorem \ref{existence-stationary}.
We will call them stationary pairs, or for brevity
{\em s-pairs}.
The set $\E$ of s-pairs is parametrized by
$\lambda, c, w_0, y_0, v$, which add up to $4n$
real parameters independent of the dimension of $M$.
Using this parametrization, we can identify $\E$ with an
open set in $\R^{4n}$.
Note that $\dim T(N^*(M))=4n$, which suggests that
$\E$ can be parametrized by $T(N^*(M))$ using the values
and derivatives of the discs at a boundary point.

We recall the following evaluation maps \cite{ST2007, T2001}:
\begin{align*}
&\F:\E\to T(N^*(M)), &&
\F(\phi,\phi^*)
=(\phi(1),\phi^*(1),J\phi'(1),J{\phi^*}'(1)), \\
&\G:\E\to N^*(M)\times N^*(M), &&
\G(\phi,\phi^*)
=(\phi(1), \phi^*(1), \phi(\zeta_0), \phi^*(\zeta_0)).
\end{align*}
Here $\zeta_0\in b\D$ is a fixed point, $\zeta_0\ne 1$.

\begin{thm}\label{Parametrization-TN*}
\cite{ST2007, T2001}
Let $M$ be $C^4$-smooth Levi generating strongly
pseudoconvex manifold in $\C^n$.
Let $(\phi,\phi^*)\in \E$.
Suppose $\phi$ is {\em not defective}.
Then $\F$ and $\G$ are respectively
$C^1$- and $C^2$-diffeomorphisms of a neighborhood of $(\phi,\phi^*)$ in $\E$
onto open sets in the target manifolds.
\end{thm}

For brevity, we call an s-pair $(\phi,\phi^*)$
non-defective if $\phi$ is not defective.
By Theorem \ref{Parametrization-TN*} every s-pair
$(\phi,\phi^*)$
close to a given non-defective s-pair is uniquely
determined by the data $\F(\phi,\phi^*)$.

\begin{cor}\label{boundaries-cover-open-set}
Under the hypotheses of Theorem \ref{Parametrization-TN*},
the union of the sets $\psi(b\D)$ for
all s-pairs $(\psi,\psi^*)$ close to $(\phi,\phi^*)$
with fixed
$(\psi,\psi^*)(1)=(\phi,\phi^*)(1)$
covers an open set
in $M$.
\end{cor}

\begin{proof}
Immediate by
the properties of the mapping $\G$.
\end{proof}

Thus the question on the parametrization of $\E$
by $T(N^*(M))$ depends on the existence of
non-defective stationary discs.
We will prove the following.

\begin{thm}\label{existence-non-defective-quadric}
  Let $M_0$ be a strongly pseudoconvex Levi generaing
  quadric in $\C^n$, that is, $M_0$ is defined by
  (\ref{M-rho-h}, \ref{hjAj})
  in which the ``O'' term is 0. Then there exists a
  {\em non-defective} s-pair for $M_0$.
  The set of all defective s-pairs
  form a proper algebraic set in the space of
  parameters.
\end{thm}

Since every manifold $M$ is locally approximated by
the quadric $M_0$,
Theorem \ref{existence-non-defective-quadric}
implies the abundance
of non-defective discs for a strongly pseudoconvex
Levi generating manifold
(see \cite{T2004}, Proposition 8.4).

We now prove the main result on 2-jet determination
assuming Theorem \ref{existence-non-defective-quadric}
and the parametrization of $\E$ by $T(N^*(M))$
provided by Theorem \ref{Parametrization-TN*}.

\begin{proof}[Proof of Theorem \ref{Main}]
Let $f_\nu:M_1\to M_2$, $\nu=1,2$, be germs of $C^2$
CR diffeomorphisms with the same 2-jets at $0\in M_1$.
Then $f=f_2^{-1}\circ f_1$ is a germ of
a CR diffeomorphism $M\to M$, here $M=M_1$.
The 2-jet of $f$ at $0$ is the 2-jet of the identity.
We show that $f=\id$.

We assume that $M$ is defined by the equations
(\ref{M-rho-h}, \ref{hjAj}), and $M_0$ is the corresponding quadric obtained from the equations of $M$ by
dropping the ``O'' term.

We use Pinchuk's \cite{Pinchuk} scaling method.
We change coordinates by $z=t^2\tilde z, w=t\tilde w$,
where $t>0$ is a small parameter.
In the new coordinates, after dropping the tildes,
the equation of $M_0$ will remain the same,
whereas $M$ will be $C^4$-close to $M_0$ in a fixed
neighborhood of 0.

Let $(\phi_0, \phi_0^*), \phi_0(1)=0$, be a non-defective
s-pair for $M_0$ provided by Theorem \ref{existence-non-defective-quadric}.
Let $(\phi, \phi^*)$ be a s-pair for $M$
constructed by Theorem \ref{existence-stationary}
using the same parameters as $(\phi_0, \phi_0^*)$,
that is,
$(\phi, \phi^*)=\F\circ\F_0^{-1}(\phi_0, \phi_0^*)$.
Here $\F$ and $\F_0$ are the parametrization maps
for $M$ and $M_0$ respectively.
Since $M$ is $C^4$-close to $M_0$,
the s-pair $(\phi, \phi^*)$ is $C^2$-close to
$(\phi_0, \phi_0^*)$.
Hence $(\phi, \phi^*)$ is non-defective
(\cite{T2004}, Proposition 8.4).

Since the 2-jet of $f$ at 0 is the identity,
after scaling, $f$ is $C^2$-close to the identity.
Likewise, $f_*:N^*(M)\to N^*(M)$ is $C^1$-close
to the identity.

Let  $(\psi, \psi^*)=(f\circ\phi, f_*\circ\phi^*)$.
Then  $\psi$ is $C^2$-close to $\phi$ and
$\psi^*|_{b\D}$ is $C^1$-close to $\phi^*|_{b\D}$.
By the structure of the lifts (Proposition \ref{structure-lifts}), it follows that the
the parameters $(\lambda, c)$ for $(\psi, \psi^*)$
are close to those for $(\phi, \phi^*)$, in particular,
$\psi^*|_{b\D}$ is $C^2$-close to $\phi^*|_{b\D}$.
Then $(\psi, \psi^*)\in\E$, the set of s-pairs
provided by Theorem \ref{existence-stationary}.

Since the 2-jet of $f$ at $0$ is the identity,
we have $\F(\psi, \psi^*)=\F(\phi, \phi^*)$.
Since $\F$ is injective in a neighborhood of
$(\phi, \phi^*)$, we have $\psi=\phi$.
Then $f|_{\phi(b\D)}=\id$.

This argument applies to all s-pairs close to
$(\phi, \phi^*)$. Then by Corollary \ref{boundaries-cover-open-set}, $f$ is the identity
on an open set in $M$. Hence $f=\id$, as desired.
\end{proof}

\section{Existence of non-defective stationary discs}

Let $M_0\subset\C^n$ be again a strictly pseudoconvex
Levi generating quadric defined by
$$
x_j=\<A_j w,\bar w\>, \quad 1\le j\le k.
$$
Let $\lambda\in\C^k$, $c\in\R^k$ satisfy
$$
\sum\Re(\lambda_j\zeta+c_j)A_j>0,\quad
\zeta\in b\D.
$$
Let
$P=\sum \lambda_j A_j$, $Q=\sum c_j A_j$.
Then (see \cite{T2001}) the equation
\begin{equation}\label{X-quadratic}
  P^*X^2+2QX+P=0
\end{equation}
has a unique solution $X$ with all eigenvalues in $\D$.
For $v\in \C^m$ we put
$$
S(X,v)=\Span_\C \{X^\l v: \l=0,1,\dots\}.
$$
\begin{lemma}\cite{T2004}
Let $(\phi,\phi^*)$ be a stationary pair
constructed for the quadric $M_0$ by
Theorem \ref{existence-stationary}
using the data
$(\lambda,c,v)$.
The disc $\phi$ is {\em defective}
iff the linear operators $S\to\C^m$
defined by the matrices $A_1,\dots,A_k$
are linearly dependent over $\R$. Here
$S=S(X,v)$.
\end{lemma}

\begin{remark}\label{defective-exist}
{\rm
In general, defective stationary discs with condition
$\sum\Re(\lambda_j\zeta+c_j)A_j>0$, $\zeta\in b\D$,
do exist for $M_0$.
Say, let $\lambda=0$. Then $X=0$ and
$S=S(X,v)=\Span\{v\}$.
If $k>2m$, then the vectors
$A_jv\in \C^m$ are linearly dependent over $\R$,
so all discs with
data $\lambda=0$ and arbitrary $c, v$ are defective.
}
\end{remark}


Turning to the proof of
Theorem \ref{existence-non-defective-quadric},
the quadratic equation (\ref{X-quadratic}) for $X$
is hard to solve explicitly, but we can use the approximation
$$
X=-\frac{1}{2}Q^{-1}P+O(|\lambda|^3).
$$

Generically, this matrix $X$ will have distinct
eigenvalues. Then by Vandermonde,
$S(X,v)=\C^m$ for some $v$, and the conclusion follows.

In general,
we can choose $c\in\R^k$ such that $Q>0$.
Note that the criterion for defective discs
is independent of linear coordinate changes
in $\C^m$. We choose the w-coordinates
so that $Q=I$.
Then Theorem \ref{existence-non-defective-quadric}
reduces to the following.
\begin{thm}
\label{Main-New-Linear-Algebra}
Let $A_1,\dots,A_k$ be linearly independent
Hermitian $m\times m$ matrices.
Then there exist $\lambda\in\R^k$, $v\in\C^m$
such that
the linear operators $S\to\C^m$
defined by the matrices $A_1,\dots,A_k$
are linearly dependent over $\R$.
Here $S=S(X,v)$ and $X=\sum \lambda_jA_j$.
\end{thm}

\begin{remark}
{\rm
We no longer need the strong pseudoconvexity
condition here.
We only needed it for reduction to this theorem.
We were able to do the reduction in the strongly
pseudoconvex case.
}
\end{remark}

\section{Proof of Theorem \ref{Main-New-Linear-Algebra}}

We say that a $m\times m$ matrix $A$ is a
$r\times r \; (m_1,\dots, m_r)$ {\em block matrix} if $m_1+\ldots+m_r=m$,
$A=(A_{ij})_{i,j=1}^r$, and each block $A_{ij}$ is a
$m_i\times m_j$ matrix.

\begin{lemma}
\label{lemma1}
Let $A$ and $B$ be Hermitian
$2\times 2\; (m_0,m_1)$ block matrices.
Suppose $A_{00}=0, A_{10}=0, \det A_{11}\ne0$.
Let $V_0$ be the space of the first $m_0$
coordinates. Suppose for every small $t\in\R$,
the matrix $A-tB$ has the eigenspace $V(t)$
close to $V_0$, $\dim V(t)=m_0$.
Then there exist $\beta,\gamma\in\R$ such that
$B_{00}=\beta I$, $B_{01}A^{-1}_{11}B_{10}=\gamma I$.
\end{lemma}
\begin{proof}
The eigenvalues and eigenvectors of an analytic
Hermitian matrix $A-tB$ are analytic functions of $t\in\R$
(see \cite{Kato}, Chapter 2, \S6, sec. 1--2).
More precisely, there are analytic functions
$\mu_j(t)\in\R$, $v_j(t)\in\C^m$, $1\le j\le m$, such that
for all $t\in\R$, $v_j(t)$ is an eigenvector of
$A-tB$ with eigenvalue $\mu_j(t)$, and the vectors $v_j(t)$
form an orthonormal basis of $\C^m$.

By the hypotheses, possibly after renumbering
$\mu_j$-s and $v_j$-s, we have
$\lambda(t):=\mu_1(t)=\ldots=\mu_{m_0}(t)$, $\lambda(0)=0$,
and $v_1(0),\ldots,v_{m_0}(0)$ form a basis of $V_0$.

Let $u_0\in V_0$. We claim that there exists an
eigenvector $u(t)$ of $A-tB$ with eigenvalue $\lambda(t)$
such that $u(t)$ is analytic for small $t$, and $u(0)=u_0$.
Furthermore, $u(t)=u_0+u_1t+u_2t^2+\ldots$, where
$u_1, u_2,\ldots\in V_1:= V_0^\bot$.
Indeed, for $1\le j\le m_0$, we put
$v_j(t)=v_j^{(0)}(t)+ v_j^{(1)}(t)$, here
$v_j^{(0)}(t)\in V_0$, $v_j^{(1)}(t)\in V_1$,
and $v_j^{(1)}(0)=0$.
Since the inverse matrix
$(v_1^{(0)}(t),\ldots,v_{m_0}^{(0)}(t))^{-1}$
is analytic for small $t\in\R$, there are analytic
functions $c_j(t)\in\C$ such that for all small $t$,
\begin{equation}\label{u0}
u_0=\sum_{j=1}^{m_0}c_j(t)v_j^{(0)}(t).
\end{equation}
We put
$u(t)=\sum_{j=1}^{m_0}c_j(t)v_j(t)$.
Then $u(t)$ is an eigenvector of $A-tB$ with eigenvalue
$\lambda(t)$, and by (\ref{u0}) we have
$u(t)=u_0+\sum_{j=1}^{m_0}c_j(t)v_j^{(1)}(t)$.
Since $v_j^{(1)}(0)=0$, the claim follows.

We now have
$(A-tB)u(t)=\lambda(t)u(t)$, that is,
$$
(A-tB-\lambda_1t-\lambda_2t^2-\ldots)
(u_0+u_1t+u_2t^2+\ldots)=0.
$$
From vanishing the coefficients of degrees 0, 1, and 2,
we obtain
\begin{align}
& Au_0=0,\notag\\
& Au_1=(B+\lambda_1 I)u_0,\label{Au1}\\
& Au_2=(B+\lambda_1 I)u_1+\lambda_2 u_0.\label{Au2}
\end{align}
Then (\ref{Au1}) implies
\begin{align}
& (B_{00}+\lambda_1I)u_0=0,\label{B00}\\
& A_{11}u_1=B_{10}u_0,\label{A11u1}
\end{align}
whereas (\ref{Au2}) implies
\begin{equation}\label{B01u1}
B_{01}u_1+\lambda_2u_0=0.
\end{equation}
From (\ref{A11u1}, \ref{B01u1}) we obtain
\begin{equation}\label{B01A11B10}
(B_{01}A_{11}^{-1}B_{10}+\lambda_2 I)u_0=0.
\end{equation}
Since $u_0\in V_0$ is arbitrary, the equations
(\ref{B00}, \ref{B01A11B10}) imply the desired
conclusion with $\beta=-\lambda_1$,
$\gamma=-\lambda_2$.
\end{proof}

We say the matrix $B$ is {\em subordinate} to $A$
if for every small $t\in\R$ the matrix $A-tB$ has
the same number of distinct eigenvalues as $A$.

\begin{lemma}
\label{lemma2}
Let $\alpha_1,\dots,\alpha_r$ be all distinct eigenvalues
of a Hermitian matrix $A$.
Let $v_1,\dots,v_r$ be the corresponding eigenvectors.
Let $B$ be a Hermitian matrix {\em subordinate} to $A$.
Suppose $Bv_j=0$ for all $1\le j\le r$.
Then $B=0$.
\end{lemma}

\begin{proof}
By applying a unitary transformation, we can assume that
$A$ is a diagonal $(m_1,\ldots,m_r)$ block matrix, where
$m_j$ is the multiplicity of $\alpha_j$,
that is, $A_{ij}=\delta_{ij}\alpha_i I$.

Let $V_j=\{u\in\C^m:Au=\alpha_j u\}$.
Then $A-\alpha_jI$ and $B$ satisfy the hypotheses of
Lemma \ref{lemma1}. By the latter, $B_{jj}=\beta_j I$.
Since $v_j\in V_j$ and $Bv_j=0$, we have $\beta_j=0$ and
$B_{jj}=0$.

Let $\tilde A_j$ be the matrix obtained from $A-\alpha_jI$
by deleting all rows and columns contained in the block
$A_{jj}$.
Let $\tilde B_j=(B_{j1},\ldots,\hat B_{jj},\ldots,B_{js})$
be the $j$-th row in the block representation of $B$
with deleted block $B_{jj}$. Here the hat means that the
block $B_{jj}$ is missing.

By Lemma \ref{lemma1}, for some $\gamma_j\in\R$,
\begin{equation}\label{tilda-BAB}
\tilde B_j \tilde A_j^{-1} \tilde B_j^*=\gamma_j I.
\end{equation}

The latter means
\begin{equation}\label{sum-BB-all}
\sum_{\ell\ne j}(\alpha_\ell-\alpha_j)^{-1}
B_{j\ell}B_{\ell j}=\gamma_j I.
\end{equation}

The hypothesis $Bv_j=0$ means that
$$
B_{\ell j}v_j=0
$$
for all $\ell$ and $j$.
By applying (\ref{sum-BB-all}) to $v_j$ we get $\gamma_j=0$.

For definiteness $\alpha_1<\ldots<\alpha_s$.
Let $j=1$. Then  (\ref{sum-BB-all}) takes the form
\begin{equation*}
\sum_{\ell=2}^r(\alpha_\ell-\alpha_1)^{-1}
B_{1\ell}B_{1\ell}^*=0,
\end{equation*}
in which all terms are non-negative semi-definite because
$\alpha_\ell-\alpha_1>0$. Hence $B_{1\ell}=0$ for all $\ell$.
Continuing this procedure successively we get
$B_{j\ell}=0$ for all $\ell$ and $j$, that is,
$B=0$ as desired.
\end{proof}

\begin{proof}[Proof of Theorem \ref{Main-New-Linear-Algebra}]
Put $A(t)=\sum t_j A_j$, $t\in\R^k$.
Let $s(t)$ be the number of distinct eigenvalues
of $A(t)$.

Let $r=\max s(t)$, $s(\lambda)=r$, $X=A(\lambda)$.

Let $\alpha_1,\dots,\alpha_r$ be distinct
eigenvalues of $X$.
Let $v_1,\dots,v_r$ be the corresponding eigenvectors.
Let $v=v_1+\ldots+v_r$. We claim that
$\lambda$ and $v$ satisfy the conclusion of
Theorem \ref{Main-New-Linear-Algebra}.

By Vandermonde,
$S=S(X,v)=\Span_\C \{v_1,\dots,v_r\}$.

Arguing by contradiction, assume $A_j$-s are linearly
dependent over $\R$ as operators $S\to\C^m$.
Then there is a nonzero
$\mu\in\R^k$, such that the matrix $B=\sum \mu_j A_j$
has the property:
$Bv_j=0$ for all $1\le j\le r$.
Since $r=\max s(t)$, the matrix $B$ is subordinate to $X$.
By Lemma \ref{lemma2}, $B=0$, which is absurd because
$A_j$-s are linearly independent.
\end{proof}


\begin{thebibliography}{99}

\bibitem{Beloshapka}
V. K. Beloshapka,
Finite-dimensionality of the group of automorphisms of a real analytic surface. (Russian)
Izv. Akad. Nauk SSSR Ser. Mat. 52 (1988), 437--442; translation in Math. USSR-Izv. 32 (1989), 443--448.

\bibitem{BBM2019}
F. Bertrand, L. Blanc-Centi, and F. Meylan,
Stationary discs and finite jet determination for non-degenerate generic real submanifolds.
Adv. Math. 343 (2019), 910--934.

\bibitem{BM2019}
F. Bertrand and F. Meylan,
Nondefective stationary discs and 2-jet determination in higher codimension.
arXiv: 1912.10034.

\bibitem{Chern-Moser}
S. S. Chern and J. K. Moser,
Real hypersurfaces in complex manifolds.
Acta Math. 133 (1974), 219--271.

\bibitem{Kato} T. Kato,
Perturbation Theory for Linear Operators.
Springer, 1980.

\bibitem{Lempert}
L. Lempert,
La m\'etrique de Kobayashi et la repr\'esentation
des domaines sur la boule.
Bull. Soc. Math. France 109 (1981), 427--474.

\bibitem{Meylan}
F. Meylan,
A counterexample to the 2-jet determination Chern-Moser theorem in higher codimension.
arXiv: 2003.11783.

\bibitem{Pinchuk}
S. I. Pinchuk,
The scaling method and holomorphic mappings.
{\sl Proc. Sympos. Pure Math. \bf 52}, Part 1, 151--161,
Amer. Math. Soc., Providence, RI, 1991.

\bibitem{ST2007}
A. Scalari and A. Tumanov,
Extremal discs and analytic continuation of product CR maps.
Mich. Math. J. 55 (2007), 25--33.

\bibitem{T1988}
A. Tumanov,
Extension of CR-functions into a wedge from a manifold of finite type. (Russian)
Mat. Sb. (N.S.) 136(178) (1988), 128--139;
translation in
Math. USSR-Sb. 64 (1989), 129--140.

\bibitem{T1994}
A. Tumanov,
Analytic discs and the regularity
of CR mappings in higher codimension.
Duke Math. J. 76 (1994), 793--807.

\bibitem{T2001}
A. Tumanov,
Extremal discs and the regularity of CR mappings in higher codimension.
Amer. J. Math. 123 (2001), 445--473.

\bibitem{T2004}
A. Tumanov,
Extremal discs and the geometry of CR manifolds.
Lect. Notes in Math. 1848 (2004), 191--212.

\end{thebibliography}
\end{document}